\documentclass[11pt,a4paper]{amsart}

\usepackage{amsfonts,amsmath,amssymb,amsthm}

\usepackage[utf8]{inputenc}
\usepackage[T1]{fontenc}

\newcommand{\Cc}{\mathbb{C}}

\newcommand{\Nn}{\mathbb{N}}
\newcommand{\Zz}{\mathbb{Z}}
\newcommand{\Qq}{\mathbb{Q}}
\newcommand{\Ff}{\mathbb{F}}

\renewcommand {\leq}{\leqslant}
\renewcommand {\geq}{\geqslant}
\renewcommand {\le}{\leqslant}
\renewcommand {\ge}{\geqslant}

\newcommand{\Res}{\mathop{\mathrm{Res}}\nolimits}

\newcommand{\lcm}{\mathop{\mathrm{lcm}}\nolimits}

\renewcommand{\ast}{\star}

\newcommand{\defi}[1]{\emph{#1}}

{\theoremstyle{plain}
\newtheorem{theorem}{Theorem}[section]    
\newtheorem{lemma}[theorem]{Lemma}       
\newtheorem{proposition}[theorem]{Proposition}      
\newtheorem{corollary}[theorem]{Corollary}      
\newtheorem*{theorem*}{Theorem}
\newtheorem*{hypothesis*}{Schinzel Hypothesis}
}
{\theoremstyle{remark}
\newtheorem*{remark*}{Remark}  
\newtheorem{example}{Example}
}

\usepackage[a4paper]{geometry}
\title{Prime and coprime values of polynomials}

\author{Arnaud Bodin}
\author{Pierre D\`ebes}
\author{Salah Najib}

\email{arnaud.bodin@univ-lille.fr}
\email{pierre.debes@univ-lille.fr}
\email{slhnajib@gmail.com}
\address{Universit\'e de Lille, CNRS, Laboratoire Paul Painlev\'e, 59000 Lille, France}
\address{Universit\'e de Lille, CNRS, Laboratoire Paul Painlev\'e, 59000 Lille, France}
\address{Laboratoire ATRES, Facult\'e Polydisciplinaire de Khouribga, Universit\'e Sultan Moulay Slimane, BP 145, Hay Ezzaytoune, 25000 Khouribga, Maroc.}



\thanks{\emph{Acknowledgment}. 
This work was supported in part by the Labex CEMPI  (ANR-11-LABX-0007-01) 
and by the ANR project ``LISA'' (ANR-17-CE40–0023-01). The first author thanks the University of British Columbia for his visit at PIMS. We thank Bruno Deschamps for his proof of Lemma \ref{lem:discrete}.}

\date{\today}

\begin{document}

\begin{abstract} The Schinzel Hypothesis is a celebrated conjecture in number theory linking polynomial values and prime numbers. In the same vein we investigate the common divisors of values $P_1(n),\ldots, P_s(n)$ of several polynomials. We deduce this coprime version of  the Schinzel Hypothesis: under some natural assumption, coprime polynomials assume coprime values at infinitely many integers. Consequences include a version ``modulo an integer'' of the original Schinzel Hypothesis, with the Goldbach conjecture, again modulo an integer, as a special case.
\end{abstract}

\maketitle



Given polynomials with integer coefficients, famous results and long-standing questions concern the {\it divisibility properties} of their values at integers, in particular their {\it primality}. The polynomial  $x^2+x+41$ which assumes prime values at $0,1,\ldots,39$ is a striking example, going back to Euler. On the other hand, the values of a nonconstant polynomial $P(x)$ cannot be all prime numbers : if $P(0)$ is a prime, then the other value $P(kP(0)) = a_d(kP(0))^d+\cdots +a_1kP(0)+ P(0)$ is divisible by $P(0)$, and, for all but finitely many $k\in \Zz$, is different from $\pm P(0)$, and hence cannot be a prime. 

\medskip

Whether a polynomial may assume infinitely many prime values is a deeper question. Even for $P(x)=x^2+1$, whether there are infinitely many prime numbers of the form $n^2+1$ with $n\in \Zz$ is out of reach. Bunyakowsky conjectured that the question always has an affirmative answer, under some natural assumption recalled below. The Schinzel Hypothesis generalizes this conjecture to several polynomials,
concluding that they should simultaneously take prime values; see Section \ref{sec:Schinzel-hypotheses}. 

\medskip

We follow this trend. Our main results are concerned with the \emph{common divisors of values} $P_1(n), \ldots, P_s(n)$ at integers $n$ of several polynomials, 
see Theorem \ref{th:Dstar} in Section \ref{sec:first} (proved in Section \ref{sec:proof-coprime}) and further complements 
in Section \ref{sec:more}. We can then investigate the \emph{coprimality of values} of polynomials. Generally speaking, we say that $n$ integers, with $n\geq 2$, are coprime if they have no common prime divisor.
While the Schinzel Hypothesis is still open, we obtain this ``coprime'' version: under some suitable assumption, \emph{coprime polynomials assume 
coprime values at infinitely many integers} (Corollary \ref{th:schinzel-coprime}).

\medskip

We deduce a  ``modulo $m$'' variant of the Schinzel Hypothesis, and versions of the Goldbach and the Twin Primes conjectures, again ``modulo $m$''; see Section \ref{sec:Schinzel-hypotheses}. A coprimality criterion for polynomials is offered in Section \ref{ssec:coprime}.  Finally, in Section 6, we discuss generalizations for which $\Zz$ is replaced by a polynomial ring.

\section{Common divisors of values and the coprimality question}
\label{sec:first}

\emph{For the whole paper, $f_1(x),\ldots,f_s(x)$ are nonzero polynomials with integer coefficients.}

\medskip

Assume that the polynomials $f_1(x),\ldots,f_s(x)$ are coprime ($s\geq 2$), i.e.\ they have no common root in $\Cc$. Interesting phenomena occur when 
considering the greatest common divisors: 
$$d_n = \gcd(f_1(n),\ldots,f_s(n)) \quad \text{ with } n\in \Zz.$$

It may happen that $f_1(x),\ldots,f_s(x)$ never assume coprime values, i.e., that none of the integers $d_n$ is $1$.
A simple example is $f_1(x)=x^2-x=x(x-1)$ and $f_2(x)=x^2-x+2$: all values $f_1(n)$ and $f_2(n)$ are even integers.
More generally for $f_1(x)=x^p-x$ and $f_2(x)=x^p-x+p$ with $p$ a prime number, all values $f_1(n)$, $f_2(n)$ are divisible by $p$, by Fermat's theorem.
Rule out these polynomials by assuming that \emph{no prime $p$ divides all 
values $f_1(n),\ldots,f_s(n)$ with $n\in \Zz$}. Excluded polynomials are well-understood:
modulo $p$, they vanish at every element of $\Zz/p\Zz$, hence are divisible by $x^p-x=\prod_{m\in\Zz/p\Zz} (x-m)$; so they are of the form $p g(x) + h(x) (x^p-x)$ with $g(x), h(x)\in \Zz[x]$ for some prime $p$.

With this further assumption, is it always true that $f_1(n),\ldots,f_s(n)$ are coprime for
at least one integer $n$? For example this is the case for $n$ and $n+2$ that are coprime when $n$ is odd. In other words, does the set
$$\mathcal{D}^\ast = \{d_n \mid n\in \Zz\}$$
contain $1$? Studying ${\mathcal D}^\ast$, which, as we will see, is quite intriguing, is a broader goal. 

\begin{example}
\label{example:intro}
Let $f_1(x)=x^2 - 4$ and $f_2(x)=x^3 + 3x + 2$. 
These polynomials are coprime since no root of $f_1$ is a root of $f_2$. 
The values $d_n = \gcd(f_1(n),f_2(n))$, for $n=0,\ldots,20$ are:
$$
2\quad 3\quad 16\quad 1\quad 6\quad 1\quad 4\quad 3\quad 2\quad 1\quad 24 \quad
1\quad 2\quad  3\quad 4\quad 1\quad 6\quad 1\quad 64\quad 3\quad 2
$$
We have in fact $\mathcal{D}^\ast = \{ 1, 2, 3, 4, 6, 8, 12, 16, 24, 32, 48, 64, 96, 192 \}$.
It seems unclear to highlight a pattern from the first terms, but at least the integer $1$ occurs.
\end{example}

A first general observation is that the set $\mathcal{D}^\ast$ is finite. This was noticed for two polynomials by Frenkel-Pelik\'{a}n \cite{FP}. In fact they showed more: the sequence $(d_n)_{n\in \Zz}$ is periodic. We will adjust their argument.
A new result about the set ${\mathcal D}^\ast$ is the stability assertion of the following statement, which is proved in Section \ref{sec:proof-coprime}.
\begin{theorem}
\label{th:Dstar}
Let $f_1(x),\ldots,f_s(x) \in \Zz[x]$ be nonzero coprime polynomials ($s\ge2$).
The sequence $(d_n)_{n\in \Zz}$ is periodic and the finite set $\mathcal{D}^\ast = \{ d_n \}_{n\in\Zz}$ is  stable under gcd and under lcm. Consequently, the gcd $d^\ast$ and the lcm $m^\ast$ of all integers $d_n$ ($n\in \Zz$) are in the set $\mathcal{D}^\ast$.
\end{theorem}

The stability under gcd means that for every $n_1,n_2\in \Zz$, there exists $n\in \Zz$ such that $\gcd(d_{n_1},d_{n_2})=d_n$. 
In Example \ref{example:intro}, the sequence $(d_n)_{n\in \Zz}$ can be checked to be periodic of period $192$ and the set $\mathcal{D}^\ast $ is indeed stable under gcd and lcm.

A consequence of Theorem \ref{th:Dstar} is the following result, proved in
\cite[Theorem 1]{Sc02}; as discussed below in Section \ref{sec:proof-coprime}, it is a ``coprime'' version of the Schinzel Hypothesis.

\begin{corollary}
\label{th:schinzel-coprime}
Assume that $s\geq 2$ and $f_1(x),\ldots,f_s(x)$ are coprime polynomials.
Assume further  that no prime number divides all integers $f_1(n),\ldots,f_s(n)$ for every $n\in \Zz$. Then there exist infinitely many $n\in \Zz$ such that $f_1(n),\ldots,f_s(n)$ are coprime integers.
\end{corollary}

In Example \ref{example:intro}, we have $f_1(1)=-3$ and $f_2(0)=2$, so no prime number divides $f_1(n)$, $f_2(n)$ for every $n\in\Zz$. Corollary \ref{th:schinzel-coprime} asserts that $f_1(n)$ and $f_2(n)$ are coprime integers for infinitely many $n \in \Zz$.

Assuming Theorem \ref{th:Dstar}, here is how Corollary \ref{th:schinzel-coprime} is deduced.

\begin{proof}
The integer $d^\ast$, defined as the gcd of all the $d_n$, is also the gcd of all values $f_1(n),\ldots,f_s(n)$ with $n\in \Zz$. The assumption of Corollary \ref{th:schinzel-coprime} exactly says that $d^\ast = 1$. By Theorem \ref{th:Dstar}, we have $1 \in \mathcal{D}^\ast$, that is: there exists $n\in \Zz$ such that $f_1(n),\ldots,f_s(n)$ are coprime.
Due to the periodicity of $(d_n)_{n\in \Zz}$, the set of such $n$ is infinite.
\end{proof}

\section{The Schinzel Hypothesis} 
\label{sec:Schinzel-hypotheses}

The Schinzel Hypothesis is the following statement; it was denoted by (H) in \cite{SS}.
\begin{hypothesis*}
Assume that $s\geq 1$ and $f_1(x),\ldots,f_s(x)$ are irreducible in $\Zz[x]$. Assume further that no prime number divides the product $\prod_{i=1,\ldots,s}f_i(n)$ for every $n\in \Zz$. Then there exist infinitely many integers $n$ such that $f_1(n),\ldots,f_s(n)$ are all prime numbers.
\end{hypothesis*}

This statement would imply many other conjectures in number theory. For instance with $f_1(x)=x$ and $f_2(x)=x+2$, it yields the Twin Primes conjecture: there exist infinitely many primes $p$ such that $p+2$ is also a prime number. It also provides infinitely many prime numbers of the form $n^2+1$ with $n\in \Zz$; see \cite{SS} and \cite[Ch.~3 and Ch.~6]{Ri} for other problems.

The Schinzel Hypothesis is however wide open. It is only known true when $s=1$ and $\deg(f_1)=1$, and this case is already quite deep. It is indeed the Dirichlet theorem: if $a$, $b$ are coprime nonzero integers, then there are infinitely many $\ell \in \Zz$ such that $a + \ell b$ is a prime number.

Corollary \ref{th:schinzel-coprime} at least provides a ``coprime'' version of the Schinzel Hypothesis. 
This coprime version can then be conjoined with the Dirichlet theorem. This yields the following.
\begin{corollary}	
\label{cor:schinzel-coprime}
Assume that $f_1(x)$ and $f_2(x)$ are coprime polynomials and that no prime number divides $f_1(n)$ and $f_2(n)$ for every $n\in \Zz$. Then, for infinitely many $n\in \Zz$, there exist infinitely many $\ell \in \Zz$ such that $f_1(n)+\ell f_2(n)$ is a prime number.
\end{corollary}

\begin{proof}
As no prime number divides $f_1(n)$ and $f_2(n)$ for every $n\in \Zz$, we can apply Corollary \ref{th:schinzel-coprime} to get infinitely many integers $n\in \Zz$ such that $f_1(n)$ and $f_2(n)$ are coprime.
By the Dirichlet theorem for primes in an arithmetic progression, for each of these $n$ except roots of $f_2$, there exist infinitely many $\ell \in \Zz$ such that $f_1(n)+\ell f_2(n)$ is a prime number. 
\end{proof}

Corollary \ref{cor:schinzel-coprime} extends to the case $s\geq 2$. Under the generalized assumption that no prime divides all $f_1(n),\ldots,f_s(n)$ for every $n\in \Zz$, the conclusion becomes: 
\emph{for infinitely many $n\in \Zz$, there exists a ``large'' \footnote{ ``large'' should be understood as \emph{Zariski dense} in $\Zz^{s-1}$; this is the generalization of ``infinite'' for a subset $\mathcal{L}\subset \Zz^{s-1}$: if a polynomial $P(x_2,\ldots,x_s)$ vanishes at every point of ${\mathcal L}$, it has to be the zero polynomial. } set  $\mathcal{L} \subset \Zz^{s-1}$ of tuples $(\ell_2,\ldots,\ell_{s})$ such that $f_1(n)+\ell_2 f_2(n)+\cdots + \ell_s f_s(n)$ is a prime number.} We leave the reader work out the generalization.

We also obtain this ``modulo $m$'' version of the Schinzel Hypothesis.
\begin{corollary}	
\label{cor:schinzel-modulo}
For $s\geq 1$, assume that no prime integer divides $\prod_{i=1,\ldots,s}f_i(n)$ for every $n\in \Zz$. Then, given any integer $m>0$, there exist $n\in \Zz$ such that each of the values $f_1(n),\ldots,f_s(n)$ is congruent to a prime number modulo $m$.
In fact, there are infinitely many integers $n$ such that for each $i=1,\ldots, s$, there are infinitely many prime numbers $p_i$ such that $f_i(n) = p_i \pmod{m}$.
\end{corollary}

\begin{proof}
Fix an integer $m>0$. Consider the two polynomials $F_1(x) = \prod_{j=1,\ldots,s}f_j(x)$ and $F_2(x) = m$. Clearly, $F_1(x)$ and $F_2(x)$ satisfy the assumptions of Corollary \ref{th:schinzel-coprime}. It follows that there exists $n\in \Zz$ such that $F_1(n)=f_1(n)\cdots f_s(n)$ is coprime with $m$.
In particular, each of the integers $f_1(n),\ldots,f_s(n)$ is coprime with $m$. Hence, by the Dirichlet theorem, there exists a prime number $p_j$ such that $p_j = f_j(n) + a_j m$ (for some $a_j\in\Zz$).
In fact the Dirichlet theorem asserts that there are infinitely many such primes $p_j$. 
For $j=1,\ldots,s$ the congruences,
$$f_j(n+\ell m) = f_j(n) \pmod{m}$$
provide the infiniteness of the integers $n$.
These congruences are easily deduced from the basic ones for which $f_j(x)$ is a monomial $x^k$; they will again be used later.
\end{proof}

Corollary \ref{cor:schinzel-modulo} has this nice special case, which can also be found in Schinzel's paper \cite{Sc59} following works of Sierpi\'nski. 	
\begin{example}[Goldbach Theorem modulo $m$] 
\emph{Let $m, \ell$ be two positive integers. Then there exist infinitely many prime numbers $p$ and $q$ such that $p+q= 2\ell \pmod{m}$.}

	\emph{Proof.} Take $f_1(x)=x$ and $f_2(x)=2\ell-x$.
	As $f_1(1)f_2(1) = 2\ell-1$ and $f_1(-1)f_2(-1) = -(2\ell+1)$, no prime number divides $f_1(n)f_2(n)$ for every $n\in \Zz$.
	By Corollary \ref{cor:schinzel-modulo}, there exist $n\in \Zz$ and prime numbers $p$ and $q$ such that $f_1(n) = n$ is congruent to $p \pmod{m}$
	and $f_2(n) = 2\ell - n$ is congruent to $q \pmod{m}$, whence $p+q=2\ell \pmod{m}$.
	
Another example with $f_1(x)=x$ and $f_2(x)=x+2$ gives the \emph{Twin Primes Theorem modulo $m$:} 
	\emph{For every $m>0$, there are infinitely many primes $p$, $q$ such that $q = p+2 \pmod{m}$.} 
\end{example}

\section{Proof of Theorem \ref{th:Dstar}} 
\label{sec:proof-coprime}

After a brief reminder in Section \ref{ssec:reminder}, Theorem \ref{th:Dstar} is proved in Sections \ref{ssec:proof-part1} and \ref{ssec:proof-part2}. Recall that $f_1(x),\ldots,f_s(x)$ are nonzero polynomials with integer coefficients.

\subsection{Reminder on coprimality of polynomials} 
\label{ssec:reminder}

Denote the gcd of $f_1(x),\ldots,f_s(x)$ in $\Qq[x]$ by $d(x)$; it is a polynomial in $\Qq[x]$, well-defined up to a nonzero multiplicative constant in $\Qq$. Polynomials $f_1(x),\ldots,f_s(x)$ are said to be \defi{coprime} if $d(x)$ is the constant polynomial equal to $1$.
These characterizations are well-known:
\begin{proposition}
\label{prop:eqcoprime}
For $s\geq 2$, the following assertions are equivalent:
\begin{itemize}
	\item[(i)] $f_1(x),\ldots,f_s(x)$ are coprime polynomials (i.e.\ $d(x)=1$),
	\item[(ii)] the gcd of $f_1(x),\ldots,f_s(x)$ in $\Zz[x]$ is a constant polynomial,	
	\item[(iii)] $f_1(x),\ldots,f_s(x)$ have no common complex roots,
	\item[(iv)] there exist $u_1(x),\ldots,u_s(x) \in \Qq[x]$ such that a Bézout identity is satisfied, i.e.:
	$$u_1(x) f_1(x) + \cdots + u_s(x) f_s(x) = 1.$$
\end{itemize}
\end{proposition}

A brief reminder: (iv) $\Rightarrow$ (iii) is obvious; 
so is (iii) $\Rightarrow$ (ii) (using that $\Cc$ is algebraically 
closed); (ii) $\Rightarrow$ (i) is an exercise based on ``removing 
the denominators'' and  Gauss's Lemma \cite[IV, \S 2]{La02}; 
and (i) $\Rightarrow$ (iv) follows from $\Qq[x]$ being a Principal Ideal Domain.

In the case of two polynomials, we have this additional equivalence:
$f_1(x)$ and $f_2(x)$ are coprime if and only if their resultant $\Res(f_1,f_2) \in \Zz$ is non-zero. Section \ref{ssec:coprime} offers an alternate method to check coprimality of two or more polynomials.

For the rest of this section, assume that $s\geq 2$ and $f_1(x),\ldots, f_s(x)$ are coprime.
Denote by $\delta$ the smallest positive integer such that there exist $u_1(x),\ldots,u_s(x) \in \Zz[x]$ with $u_1(x)f_1(x)+\cdots+u_s(x)f_s(x)=\delta$. Such an integer exists from the Bézout identity of Proposition \ref{prop:eqcoprime}, rewritten after multiplication by the denominators.

\subsection{Finiteness of ${\mathcal D}^\ast$ and periodicity of $(d_n)_{n\in \Zz}$} \label{ssec:proof-part1}

\begin{proposition} 
\label{prop:dn}
We have the following:
\begin{itemize}
	\item Every integer $d_n$ divides $\delta$ ($n\in \Zz$). In particular, the set ${\mathcal D}^\ast$ is finite.
	\item The sequence $(d_n)_{n\in \Zz}$ is periodic of period $\delta$.
\end{itemize}
\end{proposition}

Note that the integer $\delta$ need not be the smallest period.
Proposition \ref{prop:dn} is an improved version of results by Frenkel and Pelik\'{a}n \cite{FP}: for two coprime polynomials $f_1(x)$, $f_2(x)$, they show that every $d_n$ divides the resultant $\Res(f_1,f_2)$ of $f_1(x)$ and $f_2(x)$. In fact our $\delta$ divides $\Res(f_1,f_2)$. Next example shows that $\Res(f_1,f_2)$ and $\delta$ may be huge and the sequence $(d_n)_{n\in \Zz}$ may have a complex behavior despite being periodic.

\begin{example}
\label{ex:knuth}
Let $f(x) = x^8+x^6-3x^4-3x^3+x^2+2x-5$ and $g(x) = 3x^6+5x^4-4x^2-9x+21$. These two polynomials were studied by Knuth \cite[Division of polynomials, p.~427]{Kn}.
We have $\Res(f,g) = 25\,095\,933\,394$ and $\delta = 583\,626\,358 = 2 \times 7^2 \times 43 \times 138\,497$. Here are the terms $d_n$ for $0\leq n\leq 39$:
$$
 1\  2\ 1\ 2\ 7\ 2\ 1\ 2\ 1\ 2\ 1\ 14\ 1\ 2\ 1\ 2\ 1\ 2\ 7\ 2\ 
 1\ 86\ 1\ 2\ 1\ 14\ 1\ 2\ 1\ 2\ 1\ 2\ 7\ 2\ 1\ 2\ 1\ 2\ 1\ 98 
$$
Higher values  occur: for instance $d_{1999} = 4214$, $d_{133\,139} = 276\,994$. For this example, the set ${\mathcal D}^\ast$ is exactly the set of all divisors of $\delta$ and the smallest period is $\delta$.
\end{example}

\begin{proof}[Proof of Proposition \ref{prop:dn}]
The identity $u_1(n)f_1(n)+\cdots+u_s(n)f_s(n)=\delta$ implies that $d_n=\gcd(f_1(n),\ldots,f_s(n))$ divides $\delta$ ($n\in \Zz$).
To prove that the sequence $(d_n)_{n\in \Zz}$ is periodic, we use again that $f_j(n+\ell \delta) = f_j(n) \pmod{\delta}$ for every $\ell \in \Zz$ and every $n \in \Zz$.

Fix $n,\ell \in \Zz$. As $d_n$ divides $f_j(n)$ and $\delta$, then by this congruence, $d_n$ divides $f_j(n+\ell \delta)$. This is true for $j=1,\ldots,s$,  whence $d_n$ divides $d_{n+\ell\delta}$.
In the same way we prove that $d_{n+\ell\delta}$ divides $d_n$ ($n,\ell \in \Zz$). Thus $d_{n+\ell\delta}=d_n$ and $(d_n)_{n\in \Zz}$ is periodic of period  $\delta$. 
\end{proof}

\subsection{Stability by gcd and lcm} 
\label{ssec:proof-part2}

\begin{proposition}
\label{th:gcdlcm}
The set $\mathcal{D}^\ast$ is stable under gcd and lcm.
\end{proposition}

Denote by $d^\ast$ the gcd of all elements of $\mathcal{D}^\ast$ and
by  $m^\ast$ the lcm of those of $\mathcal{D}^\ast$.
Using that $\gcd(a,b,c)=\gcd(a,\gcd(b,c))$ we obtain:
\begin{corollary}
\label{cor:gcd}
The integers $d^\ast$ and $m^\ast$ are elements of ${\mathcal D}^\ast$. Furthermore $d^\ast = \min(\mathcal{D}^\ast)$ is the greatest integer dividing $f_1(n),\ldots,f_s(n)$ for every $n\in\Zz$.
Similarly $m^\ast = \max(\mathcal{D}^\ast)$.
\end{corollary}

\begin{proof}[Proof of Proposition \ref{th:gcdlcm} for the gcd]
We only prove the gcd-stability part and leave the lcm part (which we will not use) to the reader.

Let $d_{n_1}$ and $d_{n_2}$ be two elements of $\mathcal{D}^\ast$.
Let $d(n_1,n_2)$ be their gcd. The goal is to prove that $d(n_1,n_2)$ is an element of $\mathcal{D}^\ast$. The integer $d(n_1,n_2)$ can be written:
$$d(n_1,n_2) = \prod_{i\in I} p_i^{\alpha_i}$$
where, for each $i\in I$,  $p_i$ is a prime divisor of $\delta$ (see Proposition \ref{prop:dn}) and $\alpha_i \in \Nn$ (maybe $\alpha_i = 0$ for some $i\in I$).
Fix $i\in I$. As $p_i^{\alpha_i+1}$ does not divide $d(n_1,n_2)$, $p_i^{\alpha_i+1}$ does not divide $d_{n_1}$ or does not divide $d_{n_2}$;
we name it $d_{m_i}$ with $m_i$ equals $n_1$ or $n_2$. 

The Chinese remainder theorem provides an integer $n$, such that
$$n = m_i \pmod{p_i^{\alpha_i+1}} \quad \text{ for each } i\in I.$$

By definition, $p_i^{\alpha_i}$ divides $d(n_1,n_2)$, so $p_i^{\alpha_i}$ divides all $f_1(n_1),\ldots, f_s(n_1)$, $f_1(n_2),\ldots, f_s(n_2)$.
In particular $p_i^{\alpha_i}$ divides $f_1(m_i),\ldots, f_s(m_i)$, hence also $f_1(n),\ldots, f_s(n)$. Whence $p_i^{\alpha_i}$ divides $d_n$ for each $i\in I$.

Now $p_i^{\alpha_i+1}$ does not divide $f_{j_0}(m_i)$, for some $j_0 \in \{1,\ldots,s\}$.
As $f_{j_0}(n) = f_{j_0}(m_i) \pmod{p_i^{\alpha_i+1}}$, then $p_i^{\alpha_i+1}$ does not divide $f_{j_0}(n)$. Hence $p_i^{\alpha_i+1}$ does not divide $d_n$.

We have proved that $p_i^{\alpha_i}$ is the greatest power of $p_i$ dividing $d_n$, for all $i\in I$. As $d_n$ divides $\delta$, each prime factor of $d_n$ is one of the $p_i$. Conclude that $d(n_1,n_2) = d_n$.
\end{proof}

\section{More on the set ${\mathcal D}^\ast$} 
\label{sec:more} 

Further questions on the set ${\mathcal D}^\ast$ are of interest. The stability under gcd and lcm gives it a remarkable ordered structure. Can more be said about elements of ${\mathcal D}^\ast$? The smallest element $d^\ast$ particularly stands out: it is also the gcd of all values $f_1(n),\ldots,f_s(n)$ with $n\in \Zz$. Can one determine or at least estimate $d^\ast$?

\begin{proposition}
\label{prop:dstar}
Assume that $f_1(x),\ldots,f_s(x)$ are monic. Then $d^\ast$ divides each of the integers $(\deg f_1)!, \ldots, (\deg f_s)!$.
\end{proposition}

The proof relies on the following result.
\begin{lemma}
\label{lem:discrete}
Let  $f(x) = a_dx^d + \cdots + a_1x+a_0$ be a polynomial in $\Zz[x]$ of degree $d$. Fix an integer $T>0$ and fix $m \in \Zz$. If an integer $k$ divides each of $f(m), f(m+T), f(m+2T),\ldots$ then $k$ divides $a_d T^d d!$.
\end{lemma}

For $T=1$, this lemma was obtained by Schinzel in \cite{Sc57}.
If $f(x)$ is assumed to be a primitive polynomial (i.e.\ the gcd of its coefficients is $1$) and $k$ divides $f(m+\ell T)$ (for all $\ell \in \Zz$) then Bhargava's paper \cite{Bh} implies that $k$ divides $T^d d!$ (see theorem 9 and example 17 there). Moreover using a theorem of P\'olya (see \cite[theorem 2]{Bh}), in Proposition \ref{prop:dstar}, we could replace the hypothesis ``$f_j(x)$ is monic'' by ``$f_j(x)$ is primitive'' with the same conclusion on $d^\ast$.

We give an elementary proof below of Lemma \ref{lem:discrete} which was suggested to us by Bruno Deschamps. It uses
the following operator:
 
$$
\begin{array}{cccc}
\Delta :  & \Qq[x] & \longrightarrow & \Qq[x] \\
		  & P(x)   & \longmapsto     & \frac{P(x+T)-P(x)}{T}. \\
\end{array}$$
If $P(x) = a_d x^d + \cdots+ a_0$ is a polynomial of degree $d$, then $\Delta (P)(x)$ is a polynomial of degree $d-1$ of the form $\Delta (P)(x) = d a_d x^{d-1} + \cdots$ By induction, if we iterate this operator $d$ times, we obtain that $\Delta^d (P)(x) = d! a_d$ is a constant polynomial.
The polynomial $\Delta (P)(x)$ is a discrete analog of the derivative $P'(x)$. In particular $\Delta^d (P)(x) = d! a_d$ should be related to the higher derivative $P^{(d)}(x) = d! a_d$.

\begin{proof}[Proof of Lemma \ref{lem:discrete}]
The key observation  is that if $k$ divides  $f(m)$ and $f(m+T)$, then $k$ divides $T \Delta(f)(m)$.
We prove the statement by induction on the degree $d$.
\begin{itemize}
  \item For $d=0$, ``$k$ divides $f(m)$'' is exactly saying ``$k$ divides $a_0$''.
  
  \item Fix $d>0$ and suppose that the statement is true for polynomials of degree less than $d$. Let $f(x) = a_dx^d+ \cdots+a_0$ be a polynomial of degree $d$ satisfying the hypothesis. As $k$ divides $f(m+\ell T)$ for all $\ell \in \Nn$, then $k$ divides $$T \Delta(f)(m+\ell T) = f(m +(\ell+1)T)-f(m+\ell T).$$
 By induction applied to $T \Delta (f)(x) = T d a_d x^{d-1} + \cdots$, 
 the integer $k$ divides the integer $(Tda_d)  T^{d-1} (d-1)! = a_d T^d d!$.
\end{itemize}
\end{proof}

\begin{proof}[Proof of Proposition \ref{prop:dstar}]
For each $j=1,\ldots, s$, the integer $d^\ast$ divides $f_j(n)$ for every $n\in\Zz$. Thus $d^\ast$ divides $(\deg f_j)!$ by Lemma \ref{lem:discrete} (applied with $T=1$ and $a_d=1$).
\end{proof}

We can also derive a result for $m^\ast = \max(\mathcal{D}^\ast) = \lcm(\mathcal{D}^\ast)$.

\begin{proposition}
\label{prop:mstar}
Let $T$ be the smallest period of the sequence $(d_n)_{n\in\Zz}$ and $f_1(x) = a_d x^d + \cdots$ be a polynomial of degree $d$.
Then:
$$T | m^\ast \qquad \text{ and } \qquad m^\ast | a_d T^d d!$$
\end{proposition}

\begin{proof}
The proof that $m^\ast$ is a period is the same as the one for $\delta$ (see Proposition \ref{prop:dn}). It follows that $T$ divides $m^\ast$. On the other hand, if $(d_n)_{n\in \Zz}$ is periodic of period $T$, then every term
$d_n$ divides $f_1(n+\ell T)$ for all $\ell \in \Zz$. By Lemma \ref{lem:discrete}, $d_n$ divides $a_d T^d d!$.
This is true for each $n$, so $m^\ast = \lcm \{d_n\}_{n\in\Zz}$ also divides $a_d T^d d!$.
\end{proof}

\section{A coprimality criterion for polynomials}
\label{ssec:coprime}

A constant assumption of the paper has been that our polynomials $f_1(x),\ldots, f_s(x)$ are coprime. To test this condition, 
we offer here a criterion only using the values $f_1(n),\ldots, f_s(n)$ that may be more practical than the characterizations from Proposition \ref{prop:eqcoprime}.

Define the \defi{normalized height} of a degree $d$ polynomial 
$f(x)= a_dx^d+ \cdots + a_0$
by
$$H(f) = \max_{i=0,\ldots,d-1} \left| \frac{a_i}{a_d} \right|.$$

\begin{proposition}
\label{prop:coprime}
Let $H$ be the minimum of the normalized heights $H(f_1),\ldots,H(f_s)$. The polynomials $f_1(x),\ldots,f_s(x)$ are coprime if and only if there exists $n\ge 2H+3$ such that $\gcd(f_1(n),\ldots,f_s(n)) \le \sqrt{n}$.
\end{proposition}

In particular if $f_1(n),\ldots,f_s(n)$ are coprime (as integers) for some sufficiently large $n$ then $f_1(x),\ldots,f_s(x)$ are coprime (as polynomials).

\begin{example}
~
\begin{itemize}
  \item Take $f_1(x) = x^4 - 7x^3 + 3$, $f_2(x) = x^3 -3x + 3$. We have $H(f_1)=7$, $H(f_2)=3$, so $H = 3$.
For $n = 9 \, (= 2H+3)$, we have $f_1(n) = 1461$, $f_2(n) = 705$. Thus $\gcd(f_1(n),f_2(n)) = 3 \le \sqrt{n}$. From Proposition \ref{prop:coprime}, the polynomials $f_1(x)$ and $f_2(x)$ are coprime.

  \item Here is an example for which the polynomials are not coprime. Take $f_1(x) = x^2-1 = (x+1)(x-1)$, $f_2(x) = x^2+2x+1 = (x+1)^2$. Then $\gcd(f_1(x),f_2(x))=x+1$ and $\gcd(f_1(n),f_2(n)) \ge n+1$.
\end{itemize}
\end{example}

\begin{remark*}
Proposition \ref{prop:coprime} is a coprime analog of the classical idea consisting in using prime values of polynomials to prove their irreducibility. For instance there is this irreducibility criterion by Ram Murty \cite{RM}, which can be seen as a converse to the Bunyakovsky conjecture:
\emph{Let $f(x) \in \Zz[x]$ be a polynomial of normalized height $H$. If $f(n)$ is prime for some $n\ge H+2$, then $f(x)$ is irreducible in $\Zz[x]$.}
\end{remark*}

We first need a classical estimate for the localization of the roots of a polynomial, as in \cite{RM}.

\begin{lemma}[Cauchy bound]
\label{lem:root}
Let $f(x) = a_dx^d+\cdots +a_1x+a_0 \in \Zz[x]$ be a polynomial of degree $d$ and of normalized height $H$. Let $\alpha \in \Cc$ be a root of $f$. Then $|\alpha| < H+1$.
\end{lemma}

\begin{proof}[Proof of Lemma \ref{lem:root}]
We may assume $|\alpha|>1$, since for $|\alpha| \le 1$, Lemma \ref{lem:root} is obviously true.
As $f(\alpha)=0$, $\alpha$ satisfies:
$$|a_d \alpha^d| = \left| a_{d-1} \alpha^{d-1} + \cdots + a_1\alpha + a_0 \right| \le \sum_{i=0}^{d-1} \left| a_i \alpha^i \right|.$$
By dividing by $a_d$, we get:
$$|\alpha^d| \le \sum_{i=0}^{d-1} H \left| \alpha^i \right| = H \frac{|\alpha|^d-1}{|\alpha|-1} \quad \text{ then } \quad
|\alpha|-1 \le H \frac{|\alpha|^d-1}{|\alpha^d|} = H \left( 1 -\frac{1}{|\alpha|^d} \right).$$
So that $|\alpha|-1 \le H$ and the proof is over.
\end{proof}

\begin{proof}[Proof of Proposition \ref{prop:coprime}] ~
\begin{itemize}
	\item $\Longrightarrow$
	Since $f_1(x), \ldots, f_s(x)$ are coprime polynomials, we have a Bézout identity: $u_1(x)f_1(x)+\cdots+u_s(x)f_s(x)=1$
	for some $u_1(x),\ldots,u_s(x)$ in $\Qq[x]$.  By multiplying by an integer $k \in \Zz\setminus \{0\}$, we obtain $\tilde u_1(x)f_1(x)+\cdots+\tilde u_s(x)f_s(x)=k$,	with $\tilde u_1(x),\ldots,\tilde u_s(x)$ being this time in $\Zz[x]$. This gives  $\tilde u_1(n)f_1(n)+\cdots+\tilde u_s(n)f_s(n)=k$ for all $n\in\Zz$, so that $\gcd(f_1(n),\ldots,f_s(n))$ divides $k$.	
      Thus the gcd of $f_1(n),\ldots, f_s(n)$ is bounded, hence it is $\leq \sqrt{n}$ for all sufficiently large $n$.
	
	\item $\Longleftarrow$
	Let $d(x) \in \Zz[x]$ be a common divisor of $f_1(x),\ldots,f_s(x)$ in $\Zz[x]$. By contradiction, assume that $d(x)$ is not a constant polynomial. Consider an integer $n \ge 2H+3$ such that $\gcd(f_1(n),\ldots,f_s(n)) \le \sqrt{n}$.
	On the one hand $d(n)$ divides each of the $f_1(n),\ldots,f_s(n)$, so 
	$|d(n)| \le \gcd(f_1(n),\ldots,f_s(n)) \le \sqrt{n}$.
	
	On the other hand 
	$$d(n) = c\prod_{i\in I} (n-\alpha_i)$$
	for some roots $\alpha_i\in\Cc$, $i\in I$, of $f_1$ (and of the other $f_j$), and $c \in \Zz\setminus \{0\}$.
	By Lemma \ref{lem:root}, we obtain:
	$$|d(n)| = |c| \prod_{i} |n-\alpha_i| > |c| \prod_{i} |n-(H+1)| \ge  |n-(H+1)|.$$
	We obtain $|n-(H+1)| \le \sqrt{n}$, which is impossible for $n \ge 2H+3$.
	
	We conclude that the common divisors of $f_1(x),\ldots,f_s(x)$ in $\Zz[x]$ are constant. 
	Therefore by Proposition \ref{prop:eqcoprime}, the
    polynomials $f_1(x),\ldots,f_s(x)$ are coprime.
\end{itemize}  	
\end{proof}

\section{Polynomials in several variables} 
\label{sec:polynomial-rings}

The Schinzel Hypothesis and its coprime variant can be considered with the ring $\Zz$ replaced by a more general integral domain $Z$. Papers \cite{BDN19a} and  \cite{BDN19b} are devoted to this. The special case that $Z$ is a polynomial ring $\Zz[\underline u]$ stands out; here $\underline u$ can be a single variable or a tuple $(u_1,\ldots,u_r)$ of several variables. ``Prime in $\Zz[\underline u]$'' then means ``irreducible in $\Zz[\underline u]$''.

In \cite{BDN19a}, we prove the Schinzel Hypothesis for $\Zz[\underline u]$ instead of $\Zz$:
\begin{theorem}
\label{th:schinzel-polynomial}
With $s\geq 1$, let $f_1(\underline u,x),\ldots,f_s(\underline u,x)$ be $s$ polynomials, irreducible in $\Zz[\underline u,x]$, of degree $\geq 1$ in $x$.
Then there are infinitely many polynomials $n(\underline u) \in \Zz[\underline u]$
(with partial degrees as large as desired) such that 
$$f_i\big(\underline u,n(\underline u)\big)$$
is an irreducible polynomial in $\Zz[\underline u]$ for each $i=1,\ldots,s$.
\end{theorem}

We also prove the Goldbach conjecture for polynomials: \emph{any nonconstant polynomial in $\Zz[\underline u]$ is the sum of two irreducible polynomials of lower or equal degree.}
Furthermore, Theorem \ref{th:schinzel-polynomial} is shown to also hold with the coefficient ring $\Zz$ replaced by more general rings $R$, e.g.\ $R=\Ff_q[t]$. However not all integral domains are allowed. For example, with $\underline u$ a single variable, the result is obviously false with $R=\Cc$, is known to be false for $R=\Ff_q$ by a result of Swan \cite{Sw} and is unclear for $R=\Zz_p$.

In contrast, we prove in \cite{BDN19b} that the coprime analog of Theorem \ref{th:schinzel-polynomial} holds in a much bigger generality.
\begin{theorem}
\label{th:schinzel-coprime-polynomial}
Let $R$ be a Unique Factorization Domain and assume that $R[\underline u]$ is not the polynomial ring $\Ff_q[u_1]$ in a single variable over a finite field.
With $s\geq 2$, let $f_1(\underline u,x),\ldots,f_s(\underline u,x)$ be $s$ nonzero polynomials, with no common divisor in $R[\underline u,x]$ other than units of $R$.
Then there are infinitely many polynomials $n(\underline u) \in R[\underline u]$ such that
$$f_1\big(\underline u,n(\underline u)\big), \  \ldots, \   f_s\big(\underline u,n(\underline u)\big)$$
have no common divisor in $R[\underline u]$ other than units of $R$.
\end{theorem}

Theorem \ref{th:schinzel-coprime-polynomial} fails if $R[\underline u] = \Ff_q[u_1]$. Take indeed $f_1(u_1,x) = x^q -x+u_1$
and $f_2(u_1,x) = (x^q -x)^2 +u_1$. For every $n(u_1)\in \Ff_q[u_1]$, the constant term of $n(u_1)^q - n(u_1)$ is zero, so  $f_1(u_1,n(u_1))$ and $f_2(u_1,n(u_1))$ are divisible 
by $u_1$. 

\bibliographystyle{plain}
\bibliography{gcd.bib}

\end{document}